\documentclass{amsart}

\usepackage{amssymb}
\usepackage{amsmath}
\usepackage{amssymb}
\usepackage{graphicx}
\usepackage{amsmath,mathrsfs}
\usepackage{mathtools} 
\usepackage[all]{xy}
\usepackage{fancyhdr}

\input prepictex
\input pictexwd
\input postpictex

\DeclareMathAlphabet{\mathpzc}{OT1}{pzc}{m}{it} 

\theoremstyle{plain}
\newtheorem{theorem}{Theorem}
\newtheorem{lemma}{Lemma}
\newtheorem{prop}{Proposition}
\newtheorem{corollary}{Corollary}

\theoremstyle{remark}
\newtheorem{remarks}{Remarks}

\theoremstyle{definition}

\newcommand{\V}{\mathcal{V}}
\newcommand{\Q}{\mathcal{Q}}

\newcommand{\B}{\mathcal{B}}
\newcommand{\J}{\mathcal{J}}
\newcommand{\G}{\mathcal{G}}
\newcommand{\A}{\mathcal{A}}

\newcommand{\PP}{\mathbb{P}}
\newcommand{\HH}{\mathbb{H}}
\newcommand{\PV}{P_{\V}}

\newcommand{\CAT}{\mathsf{CAT}}
\newcommand{\Set}{\mathsf{Set}}
\newcommand{\Sup}{\mathsf{Sup}}
\newcommand{\Rel}{\mathsf{Rel}}
\newcommand{\Cat}{\mathsf{Cat}}
\newcommand{\Mod}{\mathsf{Mod}}
\newcommand{\Ord}{\mathsf{Ord}}
\newcommand{\Met}{\mathsf{Met}}
\newcommand{\VV}[1]{ \V\text{-}{#1}}
\newcommand{\VRel}{\VV{\Rel}}
\newcommand{\VCat}{\VV{\Cat}}
\newcommand{\VMod}{\VV{\Mod}}
\newcommand{\two}{\mathsf{2}}

\newcommand{\Sub}{\mathrm{Sub}}
\newcommand{\ob}{\mathrm{ob}}
\newcommand{\emb}{\mathrm{emb}}
\newcommand{\sym}{\mathrm{sym}}
\newcommand{\op}{\mathrm{op}}
\newcommand{\yy}{\mathpzc{y}} 
\newcommand{\mm}{\mathpzc{m}}
\newcommand{\YY}{\mathpzc{Y}}

\newcommand{\relto}{\mathrel{\mathmakebox[\widthof{$\xrightarrow{\rule{1.45ex}{0ex}}$}]
{\xrightarrow{\rule{1.45ex}{0ex}}\hspace*{-2.4ex}{\mapstochar}\hspace*{1.8ex}}}} 
\newcommand{\ihom}{\mathrel{\mathmakebox[\widthof{$-$}-3pt]
{$--$\hspace*{-0.15ex}{\circ}}}} 
\newcommand{\modto}{\mathrel{\mathmakebox[\widthof{$\xrightarrow{\rule{1.45ex}{0ex}}$}]
{\xrightarrow{\rule{1.45ex}{0ex}}\hspace*{-2.8ex}{\circ}\hspace*{1ex}}}}
\newcommand{\submodto}{
    \mathrel{
        \mathmakebox[2.7ex]{
            \longrightarrow
            \hspace*{-2.1ex}{
                \circ
            }
            \hspace*{1ex}
        }
    }
}

\newdir{|>}{%
!/4.5pt/@{|}*:(1,-.2)@^{>}*:(1,+.2)@_{>}}

\newcommand\unnumberedfootnote[1]{ %
        \let\temp=\thefootnote %
        \renewcommand{\thefootnote}{}%
        \footnote{#1}%
        \let\thefootnote=\temp%
        \addtocounter{footnote}{-1}}

\begin{document}

\title{On the categorical meaning of Hausdorff\\ and Gromov distances, I}

\author{Andrei Akhvlediani}
\address{Oxford University Computing Laboratory\\ Oxford OX1 3QD, United Kingdom} \email{andrei.akhvlediani@comlab.ox.ac.uk}
\thanks{The first author acknowledges partial financial assistance from NSERC}

\author{Maria Manuel Clementino}
\address{Department of Mathematics\\ University of Coimbra\\ 3001-454 Coimbra, Portugal}
\email{mmc@mat.uc.pt}
\thanks{The second author acknowledges financial support from the Center of Mathematics of the University of Coimbra/FCT}

\author{Walter Tholen}
\address{Department of Mathematics and Statistics\\
  York University\\
 Toronto, ON M3J 1P3, Canada
} \email{tholen@mathstat.yorku.ca}
\thanks{The third author acknowledges partial financial assistance from NSERC}

\date{\today}

\begin{abstract}
Hausdorff and Gromov distances are introduced and treated in the context of categories enriched over a commutative unital quantale $\V$. The Hausdorff functor which, for every $\V$-category $X$, provides the powerset of $X$ with a suitable $\V$-category structure, is part of a monad on $\VV\Cat$ whose Eilenberg-Moore algebras are order-complete. The Gromov construction may be pursued for any endofunctor $K$ of $\VV\Cat$. In order to define the Gromov ``distance" between $\V$-categories $X$ and $Y$ we use $\V$-modules between $X$ and $Y$, rather than  $\V$-category structures on the disjoint union of $X$ and $Y$. Hence, we first provide a general extension theorem which, for any $K$, yields a lax extension $\tilde{K}$ to the category $\VV\Mod$ of $\V$-categories, with $\V$-modules as morphisms.
\end{abstract}



\dedicatory{}

\maketitle

\section{Introduction}
\label{s:intro}

The Hausdorff metric for (closed) subsets of a (compact) metric
space has been recognized for a long time as an important concept in
many branches of mathematics, and its origins reach back even beyond
Hausdorff \cite{Ha}, to Pompeiu \cite{Po}; for a modern account, see
\cite{BBI}. It has gained renewed interest through Gromov's work
\cite{G2}. The Gromov-Hausdorff distance of two (compact) metric
spaces is the infimum of their Hausdorff distances after having been
isometrically embedded into any common larger space. There is
therefore a notion of convergence for (isometry classes of compact)
metric spaces which has not only become an important tool in
analysis and geometry, but which has also provided the key
instrument for the proof of Gromov's existence theorem for a
nilpotent subgroup of finite index in every finitely-generated group
of polynomial growth \cite{G1}.

By interpreting the (non-negative) distances $d(x,y)$ as $\hom(x,y)$ and, hence, by rewriting the conditions
$$0\geq d(x,x),\;d(x,y)+d(y,z)\geq d(x,z)\eqno{(*)}$$
as
\[k\to \hom(x,x),\;\hom(x,y)\otimes \hom(y,z)\to\hom(x,z),\]
Lawvere \cite{L} described metric spaces as categories enriched over the (small and ``thin") symmetric monoidal-closed category $\PP_+=(([0,\infty],\geq),+,0)$, and in the Foreword of the electronic ``reprint" of \cite{L} he suggested that the Hausdorff and Gromov metrics should be developed for an arbitrary symmetric monoidal-closed category $(\V,\otimes,k)$. In this paper we present notions of Hausdorff and Gromov distance for the case that $\V$ is ``thin". Hence, we replace $\PP_+$ by a commutative and unital {\em quantale} $\V$, that is: by a complete lattice which is also a commutative monoid $(\V,\otimes,k)$ such that the binary operation $\otimes$ preserves suprema in each variable. Put differently, we try to give answers to questions of the type: which structure and properties of the (extended) non-negative real half-line allow for a meaningful treatment of Hausdorff and Gromov distances, and which are their appropriate carrier sets? We find that the guidance provided by enriched category theory \cite{Ke} is almost indispensable for finding satisfactory answers, and that so-called ({\em bi-}){\em modules} (or {\em distributors}) between $\V$-categories provide an elegant tool for the theory which may easily be overlooked without the categorical environment. Hence, our primary motivation for this work is the desire for a better understanding of the true essentials of the classical metric theory and its applications, rather than the desire for giving merely a more general framework which, however, may prove to be useful as well.

Since $(*)$ isolates precisely those conditions of a metric which lend themselves naturally to the $\hom$ interpretation, a discussion of the others seems to be necessary at this point; these are:
\begin{itemize}
\item[$-$] $d(x,y)=d(y,x)$ ({\em symmetry}),
\item[$-$] $x=y$ whenever $d(x,y)=0=d(y,x)$ ({\em separation}),
\item[$-$] $d(x,y)<\infty$ ({\em finiteness}).
\end{itemize}
With the distance of a point $x$ to a subset $B$ of the metric space $X=(X,d)$ be given by $d(x,B)=\inf_{y\in B}d(x,y)$, the {\em non-symmetric Hausdorff distance} from a subset $A$ to $B$ is defined by
\[Hd(A,B)=\sup_{x\in A} d(x,B),\]
from which one obtains the classical {\em Hausdorff distance}
\[H^sd(A,B)=\max \{Hd(A,B),Hd(B,A)\}\]
by {\em enforced} symmetrization. But not only symmetry, but also separation and finiteness get lost under the rather natural passage from $d$ to $Hd$. (If one thinks of $d(x,B)$ as the travel time from $x$ to $B$, then $Hd(A,B)$ may be thought of as the time needed to evacuate everyone living in the area $A$ to the area $B$.) In order to save them one usually restricts the carrier set from the entire powerset $PX$ to the closed subsets of $X$ (which makes $H^sd$ separated), or even to the non-empty compact subsets (which guarantees also finiteness). As in \cite{HT} we call a $\PP_+$-category an {\em $L$-metric space}, that is a set $X$ equipped with a function $d:X\times X\to[0,\infty]$ satisfying $(*)$; a $\PP_+$-functor $f:(X,d)\to (X',d')$ is a non-expansive map, e.g. a map $f:X\to X'$ satisfying $d'(f(x),f(y))\leq d(x,y)$ for all $x,y\in X$. That the underlying-set functor makes the resulting category $\Met$ {\em topological} over $\Set$ (see \cite{CHT}) provides furthe!
 r evidence that properties $(*)$ are fundamental and are better considered separately from the others, even though symmetry (as a coreflective property) would not obstruct topologicity. But inclusion of (the reflective property of) separation would, and inclusion of (the neither reflective nor coreflective property of) finiteness would make for an even poorer categorical environment.

While symmetry seems to be artificially superimposed on the Hausdorff metric, it plays a crucial role for the Gromov distance, which becomes evident already when we look at the most elementary examples. Initially nothing prevents us from considering arbitrary $L$-metric spaces $X,Y$ and putting
\[GH(X,Y)=\inf_Z Hd_Z(X,Y),\]
where $Z$ runs through all $L$-metric spaces $Z$ into which both $X$ and $Y$ are isometrically embedded. But for $X=\{p\}$ a singleton set and $Y=\{x,y,z\}$ three equidistant points, with all distances $1$, say, for every $\varepsilon>0$ we can make $Z=X\sqcup Y$ a (proper) metric space, with $d(p,x)=d(x,p)=\varepsilon$ and all other non-zero distances $1$. Then $Hd_Z(X,Y)=\varepsilon$, and $GH(X,Y)=0$ follows. One has also $GH(Y,X)=0$ but here one needs non-symmetric (but still separated) structures: put $d(x,p)=d(y,p)=d(z,p)=\varepsilon$, but let the reverse distances be $1$. Hence, even {\em a posteriori} symmetrization leads to a trivial distance between non-isomorphic spaces. However, there are two ways of {\em a priori} symmetrization which both yield the intuitively desired result $\frac{1}{2}$ for the Gromov distance in this example:
 \begin{figure}[htb]
$$
\beginpicture
 \setcoordinatesystem units <.8cm,.8cm>
 \arrow <0pt> [,] from 0 0 to 4 0
 \arrow <0pt> [,] from 0 0 to 2 3.464
 \arrow <0pt> [,] from 4 0 to 2 3.464
 \arrow <0pt> [,] from 0 0 to 2 1.732
 \arrow <0pt> [,] from 4 0 to 2 1.732
 \arrow <0pt> [,] from 2 3.464 to 2 1.732
 \put {$x$} at -.24 -.15
 \put {$y$} at 4.24 -.15
 \put {$z$} at 2 3.664
 \put {$p$} at 1.8 1.8
 \put {$1$} at 2 -.25
  \put {$1$} at .85 2
  \put {$1$} at 3.15 2
  \put {$\frac{1}{2}$} at 2.2 2.5
  \put {$\frac{1}{2}$} at 1.05 1.2
  \put {$\frac{1}{2}$} at 2.95 1.2
  \endpicture
$$
\end{figure}
One way is by restricting the range of the infimum in the definition of $GH(X,Y)$ to symmetric $L$-metric spaces, which seems to be natural when $X$ and $Y$ are symmetric. (Indeed, if for our example spaces one assumes $Hd_Z(Y,X)<\frac{1}{2}$ with $d_Z$ symmetric, then the triangle inequality would be violated: $1\leq d(x,p)+d(p,y)<\frac{1}{2}+\frac{1}{2}$.) The other way ``works" also for non-symmetric $X$ and $Y$; one simply puts
\[GH^s(X,Y)=\inf_Z H^sd_Z(X,Y),\]
with $Z$ running as in $GH(X,Y)$. (When $Hd_Z(Y,X)\leq\frac{1}{2}$, then
\[\frac{1}{2}=1-\frac{1}{2}\leq\min \{d(p,x), d(p,y), d(p,z)\}=Hd_Z(X,Y)\leq H^sd_Z(X,Y),\]
and when $Hd_Z(Y,X)\geq\frac{1}{2}$, then trivially $\frac{1}{2}\leq H^sd_Z(X,Y)$.)

Having recognized $H$ (and $H^s$) as endofunctors of $\Met$, these considerations suggest that $G$ is an ``operator" on such endofunctors. But in order to ``compute" its values, one needs to ``control" the spaces $Z$ in their defining formula, and here is where modules come in. (A module between $L$-metric spaces generalizes a non-expansive map just like a relation generalizes a mapping between sets.) A module from $X$ to $Y$ corresponds to an $L$-metric that one may impose on the disjoint union $X\sqcup Y$. To take advantage of this viewpoint, it is necessary to extend $H$ from non-expansive maps to modules (leaving its operation on objects unchanged) to become a lax functor $\tilde{H}$. Hence, $GH(X,Y)$ may then be more compactly defined using an infimum that ranges just over the hom-set of modules from $X$ to $Y$.

In Section 2 we give a brief overview of the needed tools from enriched category theory, in the highly simplified context of a quantale $\V$. The purpose of Section 3 is to establish a general extension theorem for endofunctors of $\V$-$\Cat$, so that they can act on $\V$-modules rather than just on $\V$-functors. In Sections 4 and 5 we consider the Hausdorff monad of $\V$-$\Cat$ and its lax extension to $\V$-modules, and we determine the Eilenberg-Moore category in both cases. The Gromov ``distance" is considered for a fairly general range of endofunctors of $\V$-$\Cat$ in Section 6, and the resulting Gromov ``space" of isomorphism classes of $\V$-categories is presented as a large colimit. For the endofunctor $H$, in Section 7 this large ``space" is shown to carry internal monoid structures in the monoidal category $\V$-$\CAT$ which allow us to consider $H$ as an internal homomorphism. The effects of symmetrization and the status of separation are discussed in Sections 8 and 9. The fundamental question of transfer of (Cauchy-)completeness from $X$ to $HX$, as well as the question of completeness of suitable subspaces of the Gromov ``space" will be considered in the second part of this paper.

The reader is reminded that, since $\PP_+$ carries the natural $\geq$ as its order, in the context of a general quantale $\V$ the
natural infima and suprema of $\PP_+$ appear as joins ($\bigvee$) and meets ($\bigwedge$) in $\V$. While this may appear to be irritating initially, it reflects in fact the logical viewpoint dictated by the elementary case $\V=\two=(\{\bot<\top\},\wedge,\top)$, and it translates back well even in the metric case. (For example, if we write the $\sup$-metric $d$ of the real function space $C(X)$ as $d(f,g)=\displaystyle\bigwedge_{x\in X} |f(x)-g(x)|$, then the statement
\[d(f,g)=0\;\iff\;\mbox{\em for all } x\in X\;:\;f(x)=g(x)\]
seems to read off the defining formula more directly.)
\bigskip

\noindent {\em Acknowledgments} While the work presented in this paper first began to take shape when, aimed with her knowledge of the treatment of the Hausdorff metric in \cite{CH04}, the second-named author visited the third in the Spring of 2008, which then gave rise to a much more comprehensive study by the first-named author in his Master's thesis \cite{AA} that contains many elements of the current work, precursors of it go in fact back to a visit by Richard Wood to the third-named author in 2001. However, the attempt to work immediately with a (non-thin) symmetric monoidal-closed category proved to be too difficult at the time. The second- and third-named authors also acknowledge encouragement and fertile pointers given by Bill Lawvere over the years, especially after a talk of the third-named author at the Royal Flemish Academy of Sciences in October 2008. This talk also led to a most interesting exchange with Isar Stubbe who meanwhile has carried the theme of this paper into the more general context whereby the quantale $\V$ is traded for a quantaloid $\Q$ (see \cite{St}), a clear indication that the categorical study of the Hausdorff and Gromov metric may still be in its infancy.

\section{Quantale-enriched categories}
Throughout this paper, $\V$ is a commutative, unital quantale.
Hence, $\V$ is a complete lattice with a commutative, associative
binary operation $\otimes$ and a $\otimes$-neutral element $k$,
such that $\otimes$ preserves arbitrary suprema in each variable.
Our paradigmatic examples
\[
    \two = \big(\{\bot< \top\},\wedge,\top\big) \text{ and } \PP_+ = \big( ([0,
    \infty], \geq), +, 0 \big)
\]
were already considered by Lawvere \cite{L}; they serve to provide
both an order-theoretic and a metric intuition for the theory.

A {\em $\V$-relation} $r$ from a set $X$ to a set $Y$, written as $r: X
\relto Y$, is simply a function $r: X \times Y \to \V$. Its
composition with $s: Y \relto Z$ is given by
\[
    \big(s \cdot r\big)(x,z) = \bigvee_{y \in Y} r(x,y) \otimes
    s(y,z).
\]
This defines a category $\VRel$, and there is an obvious functor $\Set \to \VRel$ which assigns to a mapping $f: X \to Y$ its $\V$-graph $f_{\circ}: X \relto Y$ with $f_{\circ}(x,y) = k$ if $f(x) = y$, and $f_{\circ}(x,y) = \bot$ otherwise. This functor is faithful only if $k > \bot$, which we will assume henceforth, writing just $f$ for $f_{\circ}$. There is an involution $(\;)^{\circ}: \VRel^{\op} \to \VRel$ which sends $r: X \relto Y$ to $r^{\circ}: Y \relto X$ with $r^{\circ}(y,x) = r(x,y)$. With the pointwise order of its hom-sets, $\VRel$ becomes order-enriched, e.g. a 2-category, and mappings $f:X \to Y$ become maps in the 2-categorical sense:
\[
    1_X \leq f^{\circ} \cdot f, \;\; f \cdot f^{\circ} \leq 1_Y.
\]

A {\em $\V$-category} $X=(X,a)$ is a set $X$ with a $\V$-relation $a:X \relto X$ satisfying $1_X \leq a$ and $a \cdot a \leq a$; elementwise this means
\[
    k \leq a(x,x) \text{ and } a(x,y) \otimes a(y,z) \leq a(x,z).
\]
A {\em $\V$-functor} $f: (X,a) \to (Y,b)$ is a map $f:X \to Y$ with $f
\cdot a \leq b \cdot f$, or equivalently
\[
    a(x,y) \leq b\big(f(x), f(y)\big)
\]
for all $x,y \in X$. The resulting category $\VCat$ yields the
category $\Ord$ of (pre)ordered sets and monotone maps for $\V=\two$ and the category $\Met$ of $L$-metric spaces for $\V = \PP_+$.

$\VCat$ has a symmetric monoidal-closed structure, given by
\[
    (X,a) \otimes (Y,b) = (X \times Y, a \otimes b), \;\; X \ihom Y =
    \big(\VCat(X,Y), c\big)
\]
with
\[
    a \otimes b \big((x,y), (x',y')\big) = a(x,x') \otimes b(y,y'),
\]
\[
    c(f,g) = \bigwedge_{x \in X} b\big(f(x), g(x) \big).
\]
Note that $X \otimes Y$ must be distinguished from the Cartesian product $X \times Y$ whose structure is $a \times b$ with
\[
    a \times b\big((x,y), (x',y')\big) = a(x,x') \wedge b(y,y').
\]

$\V$ itself is a $\V$-category with its ``internal hom'' $\ihom$, given by
\[
    z \leq x \ihom y \iff z \otimes x \leq y
\]
for all $x,y,z \in \V$. The morphism $\two \to \V$ of quantales has a right adjoint $\V \to \two$ that sends $v \in \V$ to $\top$ precisely when $v \geq k$. Hence, there is an induced functor
\[
    \VCat \to \Ord
\]
which provides a $\V$-category with the order
\[
    x \leq y \iff k \leq a(x,y).
\]

Since
\[
    \VRel(X,Y) = \V^{X \times Y} = (X \times Y) \ihom \V
\]
is a $\V$-category (as a product of $(X \times Y)$-many copies of $\V$, or as a ``function space'' with $X$, $Y$ discrete), it is easy to show that $\VRel$ is $\big(\VCat\big)$-enriched, e.g.
\[
\begin{array}{rclcrcl}
  E &\to& \VRel(X,X), & &\VRel(X,Y) \otimes \VRel(Y,Z) &\to& \VRel(X,Z)\\
  * &\mapsto& 1_X, & &(r,s) &\mapsto& s \cdot r
\end{array}\]
are $\V$-functors (where $E = (\{*\}, k)$ is the $\otimes$-unit in $\VCat$).

\section{Modules, Extension Theorem}

For $\V$-categories $X = (X,a), Y = (Y,b)$ a {\em $\V$-(bi)module}
(also: $\V$-\emph{distributor}, $\V$-\emph{profunctor}) $\varphi$ from $X$ to $Y$, written as $\varphi: X \modto Y$, is a  $\V$-relation $\varphi: X \relto Y$ with $\varphi \cdot a \leq \varphi$ and $b \cdot \varphi \leq \varphi$, that is
\[
    a(x',x) \otimes \varphi(x,y) \leq \varphi(x',y) \text{ and }
    \varphi(x,y) \otimes b(y,y') \leq \varphi(x,y')
\]
for all $x,x' \in X, \; y,y' \in Y$. For $\varphi: X \modto Y$ one actually has $\varphi \cdot a = \varphi = b \cdot \varphi$, so that $1_X^* := a$ plays the role of the identity morphism in the category $\VMod$ of $\V$-categories (as objects) and $\V$-modules (as morphisms). It is easy to show that a $\V$-relation $\varphi: X \modto Y$ is a $\V$-module if, and only if, $\varphi: X^{\op} \otimes Y \to \V$ is a $\V$-functor (see \cite{CH07}); here $X^{\op}= (X, a^{\circ})$ for $X = (X,a)$. Hence,
\[
    \VMod(X,Y) = \big(X^{\op} \otimes Y\big) \ihom \V.
\]
In particular $\VMod$ \emph{is} (like $\VRel$) $\VCat$-\emph{enriched}. Also, $\VMod$ inherits the 2-categorical structure from $\VRel$, just via pointwise order.

Every $\V$-functor $f: X \to Y$ induces adjoint $\V$-modules
\[
    f_* \dashv f^* : Y \modto X
\]
with $f_* := b \cdot f$, $f^*:= f^{\circ} \cdot b$ (in $\VRel$). Hence, there are functors
\[
    (-)_*: \VCat \to \VMod, \;\; (-)^*: \VCat^{\op} \to \VMod
\]
which map objects identically. $\VCat$ becomes order-enriched (a 2-category) via
\[
    f \leq g \iff f^* \leq g^* \iff \forall x\;: f(x) \leq g(x).
\]
The  $\V$-functor $f: X \to Y$ is \emph{fully faithful} if $f^* \cdot f_* = 1_X^*$; equivalently, if $a(x,x') = b\big(f(x), f(x')\big)$ for all $x,x' \in X$.

Via
\begin{align*}
  \varphi: X &\modto Y\\[-4mm]
  \rule{1.3cm}{0.3pt}&\rule{1cm}{0.3pt}\\[-1mm]
  X^{\op}\otimes Y &\to \V\\[-4mm]
  \rule{2.2cm}{0.3pt}&\rule{2cm}{0.3pt}\\[-1mm]
  \yy_{\varphi}: Y \to (X^{\op} &\ihom \V) =: \hat{X},
\end{align*}
every $\V$-module $\varphi$ corresponds to its \emph{Yoneda mate} $\yy_{\varphi}$ in $\VCat$. In particular, $a = 1_X^*$ corresponds to the Yoneda functor
\[
    \yy_X = \yy_{1_X^*}: X \to \hat{X}.
\]
For every $\V$-functor $f: X^{\op} \to \V$ and $x \in X$ one has
\[
    1_{\hat{X}}^* (\yy_X(x), f) = f(x) \;\;\; \text{ \emph{(Yoneda Lemma)}}.
\]
In particular, $1_{\hat{X}}^*\big(\yy_X(x), \yy_X(x')\big) =
a(x,x')$, i.e. $\yy_X$ is fully faithful.

The correspondence between $\varphi$ and $\yy_{\varphi}$ gives:
\begin{prop}\label{prop:adjunction to star}
    $(-)^*: \big(\VCat\big)^{\op} \to \VMod$ has a left adjoint
    $\hat{(-)}$, given by
    \[
    \hat{\varphi}(s)(x) = \bigvee_{y \in Y} \varphi(x,y) \otimes s(y)
    \]
    for all $\varphi: X \modto Y$, $s \in \hat{Y}$, $x \in X$.
\end{prop}

\begin{proof}
    Under the correspondence
    \[
        \frac{\varphi: X \modto Y}{\Phi:Y \to \hat{X}}
    \]
    given by $\varphi(x,y) = \Phi(y)(x)$, $\Phi = 1_{\hat{X}}$ gives the unit $\eta_X: X \modto \hat{X}$ of the adjunction, with
    \[
        \eta_X(x,t) = t(x)
    \]
    for all $x \in X$, $t \in \hat{X}$. Note that one has $\eta_X = ( \yy_X)_*$, by the Yoneda Lemma. We must confirm that $\yy_{\varphi}$ is indeed the unique $\V$-functor $\Phi:Y \to \hat{X}$ with $\Phi^* \cdot \eta_X = \varphi$. But any such $\Phi$ must satisfy
    \begin{eqnarray*}
        \varphi(x,y) &= & \big( \Phi^* \cdot (\yy_X)_* \big) (x,y) \\
        & =& \bigvee_{t \in\hat{X}} 1^*_{\hat{X}} \big( \yy_X(x), t) \otimes
        1_{\hat{X}}^*\big(t, \Phi(y)\big)\\
        & \leq & \bigvee_{t \in\hat{X}} 1_{\hat{X}}^*\big(\yy_X(x),
        \Phi(y)\big) \\
        &\leq & \Phi(y)(x) \\
        & \leq & 1_{\hat{X}}^*\big(\yy_X(x), \yy_X(x)\big) \otimes
         1_{\hat{X}}^*\big(\yy_X(x), \Phi(y)\big)\\
        & \leq & \big( \Phi^* \cdot (\yy_X)_* \big) (x,y) \\
        & =& \varphi(x,y)
    \end{eqnarray*}
    for all $x \in X, y \in Y$. Hence, necessarily $\Phi =
    \yy_{\varphi}$, and the same calculation shows $\varphi =
    (\yy_{\varphi})^* \cdot \eta_X$. Now, $\hat{\varphi}: \hat{Y} \to  \hat{X}$ is the $\V$-functor corresponding to $\eta_Y \cdot \varphi$, hence
    \[
        \hat{\varphi}(s)(x) = (\eta_Y \cdot \varphi)(x,s) =
        \bigvee_{ y \in Y}\varphi(x,y) \otimes s(y),
    \]
    for all $s
    \in \hat{Y}$, $x \in X$.
\end{proof}
\begin{remarks}\label{remarks_one}
    \begin{enumerate}
     \item For $\varphi: X \modto Y$, the $\V$-functor $\hat{\varphi}$ may also be described as the left Kan extension of $\yy_{\varphi}: Y \to \hat{X}$ along $\yy_Y: Y \to \hat{Y}$.
     \item The adjunction of Proposition \ref{prop:adjunction to star} is in fact 2-categorical. It therefore induces a 2-monad     $\PP_{\V} = (P_{\V}, \yy, \mm)$ on $\VCat$, with
        \[
            P_{\V}X = \hat{X} = X \ihom \V, \;\; P_{\V}f =
            \widehat{f^*}:\hat{X} \to \hat{Y}
        \]
        for $f: X \to Y = (Y,b)$, where
        \[
            \widehat{f^*}(t)(y)= \bigvee_{x \in X} b\big(y, f(x)\big) \otimes t(x)
        \]
        for $t \in \hat{X}$, $y \in Y$. This monad is of Kock-Z\"oberlein type, i.e. one has
        \[
          \widehat{\yy_X^*} \leq \yy_{\hat{X}}: \hat{X} \to \hat{\hat{X}}.
        \]
        In fact, for all $x,y \in X = (X,a)$, and $t,s \in \hat{X}$ one has $a(x,y) \leq s(x) \ihom s(y)$, hence
        \[
            t(y) \otimes \big( t(y) \ihom a(x,y)\big) \otimes s(x) \leq
            a(x,y)\otimes s(x) \leq s(y),
        \]
        which gives
        \begin{eqnarray*}
            \widehat{\yy_X^*}(s)(t) &=& \bigvee_x \yy_X^*(t,x) \otimes s(x) \\
            & = & \bigvee_x \bigwedge_y \big( t(y) \ihom a(x,y) \big)
            \otimes s(x) \\
            & \leq & \bigwedge_y t(y) \ihom s(y) = \yy_{\hat{X}}(s)(t).
        \end{eqnarray*}
    \item  The adjunction of Proposition \ref{prop:adjunction to star} induces also a monad on $\VMod$ which we will not consider further in this paper. But see Section 5 below.
    \item  Because of (2), the Eilenberg-Moore category
        \[
            (\VCat)^{\PP_{\V}}
        \]
        has $\V$-categories $X$ as objects which come equipped with a $\V$-functor $\alpha: \hat{X} \to X$ with $\alpha \cdot \yy_X = 1_X$ and $1_{\hat{X}} \leq \yy_X \cdot \alpha$, e.g $\V$-categories $X$ for which $\yy_X$ has a left adjoint. These are known to be the $\V$-categories that have all
        \emph{weighted colimits} (see \cite{Ke}), with $\alpha$ providing a choice of such colimits. Morphisms in $(\VCat)^{\PP_\V}$ must preserve the (chosen) weighted colimits.
     \item In case $\V = \two$, $\PP_{\V}X$ can be identified with the set $P_{\downarrow}X$ of down-closed subsets of the (pre)ordered set $X$, and the Yoneda functor $X \to P_{\downarrow}X$ sends $x$ to its down-closure $\downarrow x$. Note that $P_{\downarrow}X$ is the ordinary power set $PX$ of $X$ when $X$ is discrete. $\Ord^{\PP_\downarrow}$ has complete ordered sets as objects, and its morphisms must preserve suprema. Hence, this is the category $\Sup$ of so-called  \emph{sup-lattices} (with no anti-symmetry condition).
    \end{enumerate}
\end{remarks}

    Next we prove a general extension theorem for endofunctors of $\VCat$. While maintaining its effect on objects, we wish extend any functor $K$ defined for $\V$-functors to $\V$-modules. To this end we observe that for a $\V$-module $\varphi:X \modto Y$, the left triangle of
    \[
        \xymatrix{
            & \hat{Y} \ar[dr]^{\hat{\varphi}} & & & \hat{Y} \ar[dr]^{\hat{\varphi}} & \\
            Y \ar[ur]^{\yy_Y} \ar[rr]_{\yy_{\varphi}} & & \hat{X} & Z
            \ar[ur]^{\yy_{\psi}} \ar[rr]_{\yy_{\psi \cdot \varphi}}
            & & \hat{X}
        }
    \]
    commutes, since $\yy_Y$ is the counit of the adjunction of
    Proposition \ref{prop:adjunction to star}. More generally, the right triangle commutes for every $\psi: Y \modto Z$.

\begin{theorem}[Extension Theorem]\label{theo_one}
    For every functor $K: \VCat \to \VCat$,
    \[
        \tilde{K}\varphi := \big( \xymatrix{ KX \ar[r]|{\circ}^{(K\yy_X)_*} & K\hat{X} \ar[r]|{\circ}^{(K\yy_{\varphi})^*} & KY}\big)
    \]
    defines a lax functor $\tilde{K}: \VMod \to \VMod$ which
    coincides with $K$ on objects. Moreover, if $K$ preserves full
    fidelity of $\V$-functors, the diagram
    \[
        \xymatrix {
            \VMod \ar[r]^{\tilde{K}} & \VMod \\
            (\VCat)^{\op} \ar[u]^{(-)^*} \ar[r]^{K^{\op}} &
            (\VCat)^{\op} \ar[u]_{(-)^*}\\
        }
    \]
    commutes.
\end{theorem}

\begin{proof}
    Lax functoriality of $\tilde{K}$ follows from
    \begin{eqnarray*}
        \tilde{K}(1_X^*) &=& (K\yy_X)^* \cdot (K\yy_X)_* \geq
        1_{KX}^*,\\
        \tilde{K}(\psi\cdot\varphi) &=& (K\yy_{\psi \cdot
        \varphi})^* \cdot (K\yy_X)_* \\
        &=& (K\yy_{\psi})^* \cdot (K \hat{\varphi})^* \cdot (K
        \yy_X)_* \\
        & \geq & (K\yy_{\psi})^* \cdot (K\yy_Y)_* \cdot (K
        \yy_{\varphi})^* \cdot (K\yy_X)_* \\
        &=& \tilde{K}\psi \cdot \tilde{K}\varphi,
    \end{eqnarray*}
    since $\yy_{\varphi} = \hat{\varphi} \cdot \yy_Y$ implies
    $(K\yy_{\varphi})^* = (K\yy_Y)^* \cdot (K\hat{\varphi})^*$, hence $(K\hat{\varphi})^* \geq (K\yy_Y)_* \cdot (K\yy_{\varphi})^*$ by adjunction.

    For a $\V$-functor $f:X \to Y$, the triangle
    \[
        \xymatrix{
            & Y \ar[dr]^{\yy_Y} & \\
            X \ar[ru]^f \ar[rr]_{\yy_{f^*}} & & \hat{Y}.
        }
    \]
    commutes, so that
    \[
        \tilde{K}(f^*) = (K \yy_{f^*})^* \cdot (K\yy_Y)_* = (Kf)^*(K\yy_Y)^*(K\yy_Y)_* \geq (Kf)^*,
    \]
    and one even has $\tilde{K}(f^*) = (Kf)^*$ if $K$ preserves the full fidelity of $\yy_Y$.
\end{proof}

\section{The Hausdorff Monad on $\VCat$}

Let $X = (X,a)$ be a $\V$-category. Then $\hat{X} = (X^{\op} \ihom \V) = P_{\V}X$ is closed under suprema formed in the product $\V^X$; hence, like $\V$ it is a sup-lattice. Consequently, the Yoneda functor $\yy_X: X \to \hat{X}$ factors uniquely through the free sup-lattice $PX$, by a sup-preserving
map $\YY_X:PX \to P_{\V}X$:
\[
    \xymatrix{
        X \ar[r]^{\{-\}} \ar[dr]_{\yy_X} & PX \ar[d]^{\YY_X} & B \ar@{|->}[d] \\
        & P_{\V} X & a(-, B)
    }
\]
where
\[
    a(x, B) = \bigvee_{y \in B} a(x,y)
\]
for all $x \in X$, $B \subseteq X$. We can provide the set $PX$ with a $\V$-category structure $h_X$ which it inherits from $\PV X$ (since the forgetful functor $\VCat \to \Set$ is a fibration, even a topological functor, see \cite{CHT}).
Hence, for subsets $A, B \subseteq X$ one puts
\[
    h_X(A,B) = \bigwedge_{z \in X} a(z, A) \ihom a(z, B).
\]
\begin{lemma}\label{lemma_one}
\[
    h_X(A,B) = \bigwedge_{x \in A} \bigvee_{y \in B} a(x,y).
\]
\end{lemma}
\begin{proof}
    From $k \leq a(x,A)$ for all $x \in A$ one obtains
    \[
        h_X(A,B) \leq \bigwedge_{x \in A} a(x,A) \ihom a(x, B) \leq \bigwedge_{x \in A} k \ihom a(x, B) = \bigwedge_{x \in A} a(x,B).
    \]
    Conversely, with $ v:= \displaystyle{\bigwedge_{x \in A} \bigvee_{y \in B} a(x,y)}$, we must show $v \leq a(z,A) \ihom a(z,B)$ for all $z \in X$. But since for every $x \in A$
    \[
        a(z,x) \otimes v \leq a(z,x) \otimes \bigvee_{y \in B} a(x,y) = \bigvee_{y \in B} a(z,x) \otimes a(x,y) \leq a(z,B),
    \]
    one concludes $a(z,A) \otimes v \leq a(z, B)$, as desired.
\end{proof}

For a $\V$-functor $f: X \to Y =(Y,b)$ one now concludes easily
\[
    h_X(A,B) \leq \bigwedge_{x \in A} \bigvee_{y \in B} b\big(f(x), f(y) \big) = h_Y \big(f(A), f(B)\big)
\]
for all $A,B \subseteq X$. Consequently, with
\[
    HX = (PX, h_X), \;\; Hf: HX \to HY, \;\; A \mapsto f(A),
\]
one obtains a (2-)functor $H$ which makes the diagram
\[
    \xymatrix {
        \VCat \ar[r]^H \ar[d] & \VCat \ar[d] \\
        \Set \ar[r]^P & \Set
    }
\]
commute. Actually, one has the following theorem:

\begin{theorem}\label{theo_two}
    The powerset monad $\PP = (P, \{-\}, \bigcup)$ can be lifted along the forgetful functor $\VCat \to \Set$ to a monad $\HH$ of $\VCat$ of Kock-Z\"oberlein type.
\end{theorem}

\begin{proof}
    For a $\V$-category $X$, $x \mapsto \{x\}$ gives a fully faithful $\V$-functor $\{-\}: X \to HX$. In order to show that
    \[
        \bigcup: HHX \to HX, \;\; \mathcal{A} \mapsto \bigcup\mathcal{A},
    \]
    is a $\V$-functor, it suffices to verify that for all $x \in A \in \mathcal{A} \in HHX$ and $\mathcal{B} \in HHX$ one has
    \[
        h_{HX} (\A, \B) \leq a(x, \bigcup \B).
    \]
    But for all $B \in \B$ we have
    \[
        h_X(A,B) \leq a(x,B) \leq a(x, \bigcup \B),
    \]
    so that
    \[
        h_{HX}(\A, \B) \leq \bigvee_{B \in \B} h_X(A,B) \leq a(x, \bigcup \B).
    \]
    The induced order of $HX$ is given by
    \[
        A \leq B \iff \forall x \in A: k \leq a(x,B),
    \]
    and that of $HHX$ by
    \[
        \A \leq \B \iff \forall A \in \A: \; k \leq \bigvee_{B \in \B} h_X(A,B).
    \]
    Hence, from $k \leq a(x,A) = h_X(\{x\}, A)$ for all $A \in HX$ one obtains
    \[
        \big\{ \{x\}\;|\; x \in A \big\} \leq \{A\}
    \]
    in $HHX$, which means $H\{-\}_X \leq \{-\}_{HX}$, i.e., $\HH$ is Kock-Z\"oberlein.
\end{proof}

\begin{remarks}\label{remarks_two}
    \begin{enumerate}
        \item By definition, $h_X(A,B)$ depends only on $a(-, A)$, $a(-, B)$. Hence, if we put
            \begin{eqnarray*}
                \Downarrow_X B &:=& \big\{ x \in X\;|\; \{x\} \leq B \big \} = \{x \in X \;|\; \downarrow x \leq B \}\\
                &=& \{ x \in X\;|\; k \leq a(x, B) \},
            \end{eqnarray*}
            from $B \subseteq \Downarrow_X B$ one trivially has $a(z, B) \leq a(z, \Downarrow_X B)$ for all $z \in X$, but also
            \begin{eqnarray*}
                a(z, \Downarrow_XB) &=& \bigvee_{x \in \Downarrow_X B} a(z,x) \otimes k \\
                &\leq& \bigvee_{z \in \Downarrow_X B} \bigvee_{y \in B} a(z,x) \otimes a(x,y) \leq a(z,B).
            \end{eqnarray*}
            Consequently,
            \[
                h_X(A,B) = h_X(\Downarrow_X A, \Downarrow_X B).
            \]
            This equation also implies $\Downarrow_X \Downarrow_X B = \Downarrow_X B$.
        \item $\Downarrow_X B$ of (1) must not be confused with the down-closure $\downarrow_X B$ of $B$ in $X$ w.r.t the induced order of $X$, e.g. with
            \[
                \downarrow_X B = \{x \in X\;|\; \exists y \in B\;\; x \leq y\} = \{ x \in X\;|\; \exists y \in B \; (k \leq a(x,y)) \}.
            \]
            In general, $B \subseteq \downarrow_X B \subseteq \Downarrow_X B$. While $\downarrow_XB = \Downarrow_X B$ for $\V = \two$, the two sets are generally distinct even for $\V = \PP_+$.

            \item In the induced order of $HX$ one has
            \[
                A \leq B \iff A \subseteq \Downarrow_XB.
            \]
            Hence, if we restrict $HX$ to
            \[
                H_{\Downarrow} X:= \{ B\subseteq X\;|\; B = \Downarrow_X B\},
            \]
            the induced order of $H_{\Downarrow}X$ is simply the inclusion order. $H_{\Downarrow}$ becomes a functor $H_{\Downarrow}: \VCat \to \VCat$  with
            \[
                (H_{\Downarrow}f)(A) = \Downarrow_Yf(A)
            \]
            for all $A \in H_{\Downarrow}X$, and there is a lax
            natural transformation $\iota: H_{\Downarrow} \to H$
            given by inclusion functions.
             Like $H$, also $H_{\Downarrow}$ carries a monad structure, given by
            \begin{align*}
                    X \to H_{\Downarrow}X,\;\; x \mapsto \downarrow_X x =  \Downarrow_X{x}, \\
                    H_{\Downarrow}H_{\Downarrow}X \to H_{\Downarrow}X, \;\; \B \mapsto \Downarrow_X (\bigcup \B).
            \end{align*}
            In this way $\iota: \HH_{\Downarrow} \to \HH$ becomes a
            lax monad morphism.

            \item By definition, $\yy_X$ is fully faithful. Hence, $HX$ carries  the largest $\V$-category structure making $\yy_X: HX \to \PV X$ a $\V$-functor. Equivalently, this is the largest $\V$-category structure making
                \[
                    \delta_X: X \modto HX
                \]
                with $\delta(x,B) = a(x,B)$ a $\V$-module.

            \item $\YY_X: HX \to \PV X$ defines a morphism $\HH \to \PV$ of monads. Indeed, the left diagram of
                \[
                    \xymatrix{
                        & X \ar[dl]_{\{-\}} \ar[dr]^{\yy_X} & & HHX \ar[d]_{\bigcup} \ar[r]^{H\YY_X} & H\PV X \ar[r]^{\YY_{\hat{X}}}& \PV\PV X \ar[d]^{\mm_X} \\
                        HX \ar[rr]_{\YY_X} & & \PV X & HX \ar[rr]_{\YY_X}  & & \PV X
                    }
                \]
                commutes trivially, and for the right one first observes that $\mm_X: \hat{\hat{X}} \to \hat{X}$ is defined by
                \[
                    \mm_X(\tau)(x) = \widehat{\eta_X}(\tau)(x) = \bigvee_{t \in \hat{X}} t(x) \otimes \tau(t)
                \]
                for all $\tau \in \hat{\hat{X}}$, $x \in X$. Hence, for $\B \in HHX$ we have:
                \begin{eqnarray*}
                    (\mm_X \cdot \YY_{\hat{X}} \cdot H\YY_X)(\B)(x) &=& \bigvee_{t \in \hat{X}} t(x) \otimes \YY_{\hat{X}}\big(\YY_X(\B)\big) \\
                    & = & \bigvee_{t \in \hat{X}} t(x) \otimes 1_{\hat{X}}^*\big(t, \YY_X(\B)\big)\\
                    &=& \bigvee_{t \in \hat{X}} \bigvee_{B \in \B} t(x) \otimes \big( \bigwedge_{x'\in X} t(x') \ihom a(x', B) \big) \\
                    &\leq& \bigvee_{B \in \B} \bigvee_{t \in \hat{X}} t(x) \otimes \big( t(x) \ihom a(x,B) \big)\\
                    & = & \bigvee_{B \in \B} a(x,B) \\
                    &=& \YY_X (\bigcup \B)(x) \\
                    & \leq & a(x,x) \otimes \bigvee_{B \in \B} 1_{\hat{X}}^*\big(\yy_X(x), a(-, B) \big) \\
                    & \leq& \bigvee_{t \in \hat{X}} t(x) \otimes 1_{\hat{X}}^* \big(t, \YY_X(\B)\big) \\
                    &=&(\mm_X \cdot \YY_{\hat{X}} \cdot H\YY_X)(\B)(x).
                \end{eqnarray*}
                Consequently, there is an induced algebraic functor
                \[
                    (\VCat)^{\HH} \to (\VCat)^{\PP_V}
                \]
                of the respective Eilenberg-Moore categories.
    \end{enumerate}
\end{remarks}

We briefly describe the Eilenberg-Moore category
\[
    (\VCat)^{\HH}
\]
where objects $X \in \VCat$ come equipped with a $\V$-functor $\alpha: HX \to X$ satisfying $\alpha \cdot \{-\} = 1_X$ and $1_{HX} \leq \{-\}\cdot \alpha$ (since $\HH$ is Kock-Z\"oberlein). Hence, $\alpha(\{x\}) = x$ for all $x \in X$, and $A \leq \{\alpha(A)\}$ for $A \in HX$, that is:
\[
    k \leq h_X(A, \{\alpha(A) \}) = \bigwedge_{x \in A} a\big(x, \alpha(A)\big).
\]
Consequently, $\alpha(A)$ is an upper bound of $A$ in the induced order of $X$, and for any other upper bound $y$ of $A$ in $X= (X,a)$ one has
\[
     k \leq \bigwedge_{x \in A} a(x,y) = h_X(A, \{y\}) \leq a\big(\alpha(A), \alpha(\{y\})\big) = a\big(\alpha(A), y \big)
\]
since $\alpha$ is a $\V$-functor. Hence, $\alpha(A)$ gives (a choice of) a supremum of $A$ in $X$. Moreover, the last computation shows
\begin{equation}\label{eq: definition of H-algebras}
    a(\bigvee A, y) = \bigwedge_{x \in A} a(x,y)
\end{equation}
for all $y \in X, A \in HX$ (since ``$\leq$'' holds trivially). Conversely, any $\V$-category $X = (X,a)$ which is complete in its induced order and satisfies (\ref{eq: definition of H-algebras}) is easily seen to be an object of $(\VCat)^{\HH}$.
\begin{corollary}\label{coro_one}
    The Eilenberg-Moore category of $\HH$ has order-complete $\V$-categories $X = (X,a)$ satisfying (\ref{eq: definition of H-algebras}) as its objects, and morphisms are $\V$-functors preserving (the chosen) suprema.
\end{corollary}

\section{The lax Hausdorff monad on $\VMod$}

When applying Theorem \ref{theo_one} to the Hausdorff functor $H: \VCat \to \VCat$ of Theorem \ref{theo_two} we obtain a lax functor $\tilde{H}: \VMod \to \VMod$ whose value on a $\V$-module $\varphi: X \modto Y$ may be easily computed:
\begin{lemma}\label{lemma_two}
\[
    \tilde{H}\varphi(A,B) = \bigwedge_{x \in A} \bigvee_{y \in B} \varphi(x,y)
\]
for all subsets $A \subseteq X, B\subseteq Y$.
\end{lemma}
\begin{proof}
    \begin{eqnarray*}
        \tilde{H}\varphi(A,B) &=& \bigvee_{D \in H\hat{X}} (H\yy_X)_*(A, D)
        \otimes (H\yy_{\varphi})^*(D,B) \\
        &=& \bigvee_{D \in H\hat{X}} h_{\hat{X}} \big(\yy_X(A), D\big) \otimes h_{\hat{X}}\big(D, \yy_{\varphi}(B)\big)\\
        & \leq & h_{\hat{X}} \big( \yy_X (A), \yy_{\varphi}(B) \big)\\
        &=& \bigwedge_{x \in A} \bigvee_{y \in B} 1_{\hat{X}}^*\big(\yy_X(x), \yy_{\varphi}(y) \big)\\
        &=& \bigwedge_{x \in A} \bigvee_{y \in B} \varphi(x,y) \;\;\;\; (Yoneda) \\
        &=& h_{\hat{X}} \big( \yy_X(A), \yy_{\varphi}(B) \big) \\
        & \leq & h_{\hat{X}} \big( \yy_X(A), \yy_X(A)\big) \otimes h_{\hat{X}}\big(\yy_X(A), \yy_{\varphi}(B)\big)\\
        & \leq & \bigvee_{D \in H\hat{X}} h_{\hat{X}} \big(\yy_X(A), D \big) \otimes h_{\hat{X}} \big(D, \yy_{\varphi}(B)\big) \\
        & = & \tilde{H} \varphi(A,B).
    \end{eqnarray*}
\end{proof}
We now prove that $\tilde{H}$ carries a lax monad structure.
\begin{theorem}\label{theo_three}
    $\tilde{H}$ belongs to a lax monad $\tilde{\HH} = (\tilde{H}, \delta, \nu)$ of $\VMod$ such that $\HH$ of Theorem \ref{theo_two} is a lifting of $\tilde{\HH}$ along $(-)_*: \VCat \to \VMod$.
\end{theorem}
\begin{proof}
    Let us first note that $H$ is a lifting of $\tilde{H}$ along $(-)_*$, in the sense that
    \[
        \xymatrix{
            \VCat \ar[r]^H \ar[d]_{(-)_*} & \VCat \ar[d]^{(-)_*} \\
            \VMod \ar[r]^{\tilde{H}} & \VMod
        }
    \]
    commutes. Indeed, for $f: X \to Y=(Y,b)$ in $\VCat$ and $A \in HX$, $B \in HY$ one has
    \[
        \tilde{H}(f_*)(A,B) = \bigwedge_{x \in A} \bigvee_{y \in B}b\big(f(x),y\big) = h_Y\big(f(A), B) \big) = (Hf)_* (A,B).
    \]
    The unit of $\tilde{\HH}$, $\delta:1 \to \tilde{H}$, is defined by
    \[
        \delta_X = \{-\}_*: X \modto HX, \;\; \delta_X(x,B) = h_X(\{x\}, B) = a(x,B),
    \]
    for $X = (X,a)$, $x \in X$, $B \in HX$ (see also Remarks \ref{remarks_two} (2)), and the multiplication $\nu: \tilde{H}\tilde{H} \to \tilde{H}$ can be given by
    \[
        \nu_X = \bigcup{}_*: HHX \modto HX, \;\; \nu_X(\A, B) = h_X(\bigcup\A, B) = \bigwedge_{A \in \A} h_X(A,B),
    \]
    for $\A \in  HHX$, $B \in HX$. The monad conditions hold \emph{strictly} for $\tilde{\HH}$, because they hold strictly for $\HH$. For example, $\nu \cdot \tilde{H}\delta = 1$ follows from
    \[
        \nu_X \cdot \tilde{H}\delta_X = \bigcup{}_* \cdot \tilde{H}(\{-\}_*) = \bigcup{}_* \cdot (H\{-\})_* = (\bigcup \cdot H\{-\})_* = 1_X^*.
    \]
    Surprisingly though, also the naturality squares for both $\delta_X$ and $\nu_X$ commute \emph{strictly}. Indeed, for $\varphi: X \modto Y = (Y,b)$, $x \in X$, $B \in HY$ and $\A \in HHX$ one has:
    \begin{eqnarray*}
        (\tilde{H}\varphi \cdot \delta_X)(x,B) &=& \bigvee_{A \in HX} \delta_X(x,A) \otimes \tilde{H}\varphi(A,B) \\
        & = & \bigvee_{A \in HX} h_X(\{x\}, A) \otimes \tilde{H}\varphi(A,B) \\
        & = & \tilde{H} \varphi(\{x\}, B) \\
        & = & \bigvee_{y \in B} \varphi(x,y) \\
        & = & \bigvee_{y \in B} \bigvee_{z \in Y} \varphi(x,z) \otimes b(z,y) \\
        &= & \bigvee_{z \in Y} \varphi(x,z) \otimes \big( \bigvee_{y \in B} b(z,y) \big)\\
        & = & \bigvee_{z \in Y} \varphi(x,z) \otimes \delta_Y(z,B) \\
        & = & (\delta_Y \cdot \varphi)(x, B), \\
        (\tilde{H}\varphi \cdot \nu_X)(\A, B) &= &\bigvee_{A \in HX} \nu_X(\A, A) \otimes \tilde{H} \varphi(A,B) \\
        & = & \bigvee_{A \in HX} h_X(\bigcup\A, A) \otimes \tilde{H}\varphi(A,B) \\
        & = & \tilde{H}\varphi(\bigcup \A, B) \\
        & \leq & \big( \bigwedge_{A \in \A} \bigvee_{B' \in HB} \tilde{H}\varphi(A,B') \big) \otimes \bigwedge_{B' \in HB} h_Y(B',B)\\
         &&(\text{since } k \leq h_Y(B',B) \text{ for } B' \in HB)\\
        & \leq & \bigvee_{\B \in HHY} \big( \bigwedge_{A \in \A} \bigvee_{B' \in \B} \tilde{H}\varphi (A,B') \big) \otimes \big( \bigwedge_{B' \in \B} h_Y(B', B) \big)\\
        & = & (\nu_Y \cdot \tilde{H}\tilde{H}\varphi)(\A, B) \\
        &=& \bigvee_{\B \in HHY} \tilde{H}\tilde{H}\varphi(\A, \B) \otimes \nu_Y(\B, B) \\
        & \leq & \bigvee_{\B \in HHY} \bigwedge_{A \in \A} \tilde{H}\varphi(A, \bigcup\B) \otimes h_Y(\bigcup\B, B) \\
        & \leq & \bigwedge_{A \in \A} \tilde{H}\varphi(A,B)\\
        &= & (\tilde{H}\varphi \cdot \nu_X)(\A, B).
    \end{eqnarray*}
\end{proof}

\begin{remarks}\label{remarks_three}
    \begin{enumerate}
        \item We emphasize that, while $\tilde{H}$ is only a lax functor, this is in fact the only defect that prevents $\tilde{\HH}$ from being a monad in the strict sense.

        \item In addition to the commutativity of the diagram given in the Proof of Theorem \ref{theo_three}, since $H$ obviously preserves full fidelity of $\V$-functors, from Theorem \ref{theo_one} we obtain also the commutativity of
            \[
                \xymatrix{
                    (\VCat)^{\op} \ar[r]^{H^{\op}} \ar[d]_{(-)^*} & (\VCat)^{\op} \ar[d]^{(-)^*} \\
                    \VMod \ar[r]^{\tilde{H}} & \VMod
                }
            \]
        \item If $\V$ is \emph{constructively completely distributive} (see \cite{Wood}, \cite{CH04}), then $\tilde{H}\varphi$ for $\varphi: X \modto Y$ may be rewritten as
            \[
                \tilde{H}\varphi(A,B) = \bigvee\{v \in \V\;|\; \forall x \in A\; \exists y \in B\;: \; v \leq\varphi(x,y) \}
            \]
            In this form the lax functor $\tilde{H}$ was first considered in \cite{CH04}. In the presence of the Axiom of Choice, so that $\V$ is completely distributive in the ordinary (non-constructive) sense, one can then Skolemize the last formula to become
            \[
                \tilde{H}\varphi(A,B) = \bigvee_{f: A \to B} \bigwedge_{x \in A} \varphi(x, f(x));
            \]
            here the supremum ranges over arbitrary set mappings $f:A \to B$. Hence, the $\bigwedge\bigvee$-formula of Lemma \ref{lemma_two} has been transcribed rather compactly in $\bigvee \bigwedge$-form.
    \end{enumerate}
\end{remarks}
For the sake of completeness we determine the Eilenberg-Moore algebras of $\tilde{\HH}$, i.e., those $\V$-categories $X = (X,a)$ which come equipped with a $\V$-module $\alpha: HX \modto X$ satisfying
\begin{align}
    \alpha \cdot \delta_X = 1_X^* (=a) \label{eq: equation a}\\
    \alpha \cdot \nu_X = \alpha \cdot \tilde{H}\alpha  \label{eq: equation b}
\end{align}
The left-hand sides of those equations are easily computed as
\begin{eqnarray*}
    (\alpha \cdot \delta_X) (x,y) &=& \bigvee_{B \in HX} \delta_X(x,B) \otimes \alpha(B,y) \\
    & = & \bigvee_{B \in HX} h_X(\{x\}, B) \otimes \alpha(B,y) \\
    &=& \alpha(\{x\}, y),\\
    (\alpha \cdot \nu_X) (\A, y) &=& \bigvee_{B \in HX} \nu_X (\A, B) \otimes \alpha(B,y) \\
    &=& \bigvee_{B \in HX} h_X(\bigcup\A, B) \otimes \alpha(B, y) \\
    &=& \alpha( \bigcup\A, y),
\end{eqnarray*}
for all $x,y \in X$, $\A \in HHX$. Furthermore, if $k \leq \alpha (\{x\},x)$, for all $x \in X$, then
\begin{eqnarray*}
    \alpha(\bigcup\A, y) &\leq& \bigwedge_{A \in \A} \alpha(A,y) \\
    &=& \tilde{H}\alpha (\A, \{y\}) \otimes k\\
    & \leq & \tilde{H} \alpha (\A, \{y\}) \otimes \alpha(\{y\}, y) \\
    & \leq & \bigvee_{B \in HX} \tilde{H} \alpha(\A, B) \otimes \alpha(B,y) \\
    &=& \alpha \cdot \tilde{H}\alpha(\A,y).
\end{eqnarray*}
Consequently, (\ref{eq: equation a}) and (\ref{eq: equation b}) imply $\alpha( \{x\}, y) = a(x,y)$ and then
\begin{eqnarray*}
    \alpha(A,y) &=& \alpha\Big(\bigcup\big\{\{x\} \;|\; x \in A \big\}, y\Big) \\
    &=& \bigwedge_{x \in A} \alpha(\{x\}, y)\\
    & = & h_X(A, \{y\}) = \{-\}^*(A,y)
\end{eqnarray*}
for all $A \in HX$, $y \in X$. Hence, necessarily $\alpha = \{-\}^*$; conversely, this choice for $\alpha$ satisfies (\ref{eq: equation a}) and (\ref{eq: equation b}).
\begin{corollary}\label{coro_two}
    The category of strict $\tilde{\HH}$-algebras and lax homomorphisms is the category $\VMod$ itself.
\end{corollary}
\begin{proof}
    A lax homomorphism is, by definition, a $\V$-module $\varphi: X \modto Y$ with $\varphi \cdot \alpha \leq \beta \cdot \tilde{H} \varphi$ (where $\alpha, \beta$ denote the uniquely determined structures of $X$, $Y$, respectively). A straightforward calculation shows that every $\V$-module satisfies this inequality.
\end{proof}

\section{The Gromov structure for $\V$-categories}

    With $\tilde{H}$ as in Section 5, one defines
    \[
        GH(X,Y) := \bigvee_{\varphi: X \submodto Y} \tilde{H} \varphi(X,Y)
    \]
    for all $\V$-categories $X$ and $Y$. Since for $\V$-functors $f:X' \to X$ and $g: Y' \to Y$ one has
    \[
        (g^* \cdot \varphi \cdot f_*) (x', y') = \varphi\big(f(x'), g(y')\big)
    \]
    for all $x' \in X, y' \in Y$, with Lemma \ref{lemma_two} one obtains immediately
    \[
        GH(X,Y) = GH(X',Y')
    \]
    whenever $f$, $g$ are isomorphisms.

    \begin{prop}\label{prop_two}
        $GH$ is a $\V$-category structure for isomorphism classes of $\V$-categories.
    \end{prop}
    \begin{proof}
        Clearly
        \[
            k \leq 1_{HX}^*(X,X) \leq \tilde{H}1_X^* (X,X) \leq GH(X,X),
        \]
        and
        \begin{eqnarray*}
            GH(X,Y) \otimes GH(Y,Z) & = & \bigvee_{\varphi: X \submodto Y, \psi: Y \submodto Z} \tilde{H}\varphi(X,Y) \otimes \tilde{H}\psi(Y,Z) \\
            &\leq & \bigvee_{\varphi, \psi} \bigvee_{B \in HY} \tilde{H}\varphi(X,B) \otimes \tilde{H}\psi(B,Z) \\
            &\leq & \bigvee_{\varphi, \psi} (\tilde{H}\psi \cdot \tilde{H}\varphi) (X,Z)\\
            & \leq & \bigvee_{\varphi, \psi} \tilde{H}(\psi \cdot \varphi) (X,Z) \\
            & \leq & \bigvee_{ \chi: X \submodto Z} \tilde{H}{\chi}(X,Z)\\
            & = & GH(X,Z).
        \end{eqnarray*}
    \end{proof}
    We observe that the proof relies on the lax functoriality of
    $\tilde{H}$, but not on the actual definition of $\tilde{H}$ or
    $H$. Hence, instead of $H$ we may consider any \emph{sublifting}
    $ K:\VCat \to \VCat$ \emph{of the powerset functor}, by which we
    mean an endofunctor $K$ with $X \in KX \subseteq HX$ such that
    the inclusion functions
    \[
        \iota_X: KX \to HX
    \]
    form a lax natural transformation, e.g., they are $\V$-functors such
    that
    \[
        f(A) = (Hf)(A) \leq (Kf)(A)
    \]
    in $HY$, for all $\V$-functors $f:X \to Y$ and $A \in KX$. (We
    have encountered an example of this situation in Remarks \ref{remarks_two}(3),
    with $K = H_{\Downarrow}$.) In this situation we may replace
    $H$  by $K$ in the proof of Proposition \ref{prop_two} except that for the
    invariance under isomorphism we invoked in Lemma \ref{lemma_two}. But this
    reference may be avoided: one easily shows that the diagrams
    \[
        \xymatrix{
            X' \ar[r]^{\yy_{X'}} \ar[d]_f & \widehat{X'} \ar[d]^{\widehat{f^*}} & \widehat{X'} & Y' \ar[l]_{\yy_{g^* \cdot \varphi \cdot f_*}} \ar[d]^g \\
            X \ar[r]_{\yy_X} & \hat{X} & \hat{X} \ar[u]^{\widehat{f_*}} & Y \ar[l]^{\yy_{\varphi}}
        }
    \]
    commute, so that
    \begin{eqnarray*}
        \tilde{K}(g^* \cdot \varphi \cdot f_*) &= &(K\yy_{g^* \cdot \varphi \cdot f_*} )^* \cdot (K\yy_{X'})_* \\
        &=& (Kg)^* \cdot (K\yy_{\varphi})^* \cdot (K\widehat{f_*})^* \cdot (K\yy_{X'})_*,
    \end{eqnarray*}
    while
    \begin{eqnarray*}
        (Kg)^*\cdot\tilde{K}\varphi \cdot (Kf)_* & = & (Kg)^* \cdot (K\yy_{\varphi})^* \cdot (K\yy_X)_* \cdot (Kf)_* \\
        &=& (Kg)^* \cdot (K\yy_{\varphi})^* \cdot (K\widehat{f^*})_* \cdot (K\yy_{X'})_*.
    \end{eqnarray*}
    When $f$ is an isomorphism, one has $f_*^{-1} = f^*$. Consequently, in this case $(K\widehat{f_*})^* = (K\widehat{f^*})_*$, and then
    \[
        \tilde{K}(g^*\cdot \varphi \cdot f_*) = (Kg)^*\cdot\tilde{K}\varphi\cdot(Kf)_*.
    \]
    Hence, when for any sublifting $K$ of $P$ we put
    \[
        GK(X,Y) := \bigvee_{\varphi: X \submodto Y} \tilde{K}\varphi(X,Y),
    \]
    we may formulate Proposition \ref{prop_two} more generally as:
    \begin{theorem}\label{theo_four}
        $GK$ makes $\G := \ob(\VCat)/\cong$ a (large) $\V$-category, for every sublifting $K: \VCat \to \VCat$ of the powerset functor.
    \end{theorem}
    The resulting $\V$-category
    \[
        \G K := (\G, GK)
    \]
    may, with slightly stronger assumptions on $K$, be characterized
    as a colimit. For that purpose we first prove:
    \begin{lemma}\label{lemma_three}
        If $K:\VCat \to \VCat$ is a 2-functor, then
        \[
            \tilde{K}(g^*\cdot \varphi\cdot f_*) = (Kg)^* \cdot
            \tilde{K}\varphi \cdot (Kf)_*
        \]
        for all $f,g, \varphi$ as above.
    \end{lemma}
    \begin{proof}
        It suffices to prove $(K\widehat{f_*})^* = (K\widehat{f^*})_*$ for all
        $\V$-functors $f: X' \to X$. But since both $K$ and the
        (contravariant) $\hat{(-)}$ preserve the order of hom-sets,
        from $f_* \dashv f^*$ in $\VMod$ we obtain $K\widehat{f^*} \dashv
        K\widehat{f_*}$ in $\VCat$. Now, since for any pair of $\V$-functors
        one has
        \[
            h \dashv g \iff g^* \dashv h^* \iff g^* = h_*,
        \]
        the desired identity follows with $h = K\widehat{f^*}$ and $g =
        K\widehat{f_*}$.
    \end{proof}
    \begin{prop}\label{prop_three}
        For any sublifting $K$ of the powerset functor preserving
        the order of hom-sets and full fidelity of $\V$-functors one
        has
        \[
            GK(X,Y) = \bigvee_{X \hookrightarrow Z\hookleftarrow Y}
            1_{KZ}^*(X,Y) = \bigvee_{X \hookrightarrow (X \sqcup Y,
            c) \hookleftarrow Y} \tilde{K}c(X,Y)
        \]
        for all $\V$-categories $X$ and $Y$.
    \end{prop}

    Here the first join ranges over all $\V$-categories $Z$ into
    which $X$ and $Y$ may be fully embedded, and the second one
    ranges over all $\V$-category structures $c$ on the disjoint
    union $X \sqcup Y$ such that $X$ and $Y$ become full
    $\V$-subcategories.

    \begin{proof}
    Denoting the two joins by $v$, $w$, respectively, we trivially
    have $w \leq v$, so that $v \leq GK(X,Y) \leq w$ remains to be
    shown. Considering any full embeddings
    \[
        \xymatrix{X \ar@{^(->}[r]^{j_X} & Z & Y \ar@{_(->}[l]_{j_Y}}
    \]
    and putting $\varphi:= j_Y^* \cdot (j_X)_* = j_Y^* \cdot
    1_Z^* \cdot (j_X)_*$, because of $K$'s 2-functoriality and
    preservation of full fidelity we obtain from Lemma \ref{lemma_three} and Theorem
    \ref{theo_one}
    \[
        \tilde{K} \varphi = (Kj_Y)^* \cdot \tilde{K}1_Z^* \cdot
        (Kj_X)_* = j_{KY}^* \cdot 1_{KZ}^* \cdot (j_{KX})_*
    \]
    and therefore
    \[
        1_{KZ}^* (X,Y) = \tilde{K}\varphi(X,Y) \leq GK(X,Y).
    \]
    Considering any $\varphi: X \modto Y$, one may define a
    $\V$-category structure $c$ on $X \sqcup Y$ by
    \[
    c(z,w) :=
    \left\{%
    \begin{array}{ll}
        1_X^*(z,w) & \hbox{if $z,w \in X$;} \\
        \varphi(z,w) & \hbox{if $z \in X, w \in Y$;} \\
        \bot & \hbox{if $z \in Y, w \in X$;} \\
        1_Y^*(z,w) & \hbox{if $z,w \in Y$.} \\
    \end{array}%
    \right.
    \]
    Then, with $Z:= (X \sqcup Y, c)$, we again have $\varphi = j_Y^* \cdot (j_X)_*$ and obtain
    \[
        \tilde{K}\varphi(X,Y) = \tilde{K}c(X,Y) \leq w.
    \]
    \end{proof}

    \begin{theorem}\label{theo_five}
        For $K$ as in Proposition \ref{prop_three}, $\G K$ is a colimit of the
        diagram
        \[
            \xymatrix{ \VCat_{\emb} \ar[r] & \VCat \ar[r]^K &\VCat
            \ar@{^(->}[r] & \VV{\CAT}}.
        \]
    \end{theorem}
    Here $\VCat_{\emb}$ is the category of small $\V$-categories
    with full embeddings as morphisms, and $\VV{\CAT}$ is the
    category of (possibly large) $\V$-categories.

    \begin{proof}
        The colimit injection $\kappa_X: KX \to \G K$ sends $A
        \subseteq X$ to (the isomorphism class of) $A$, considered as
        a $\V$-category in its own right. Since for $A,B \in KX$ one
        has full embeddings $A \hookrightarrow X, B \hookrightarrow X$, trivially
        \[
            1_{KX}^* \leq GK(A,B).
        \]
        Hence $\kappa_X$ is a $\V$-functor, and $\kappa =
        (\kappa_X)_X$ forms a cocone. Any cocone given by $\V
        $-functors $\alpha_X: KX \to (\J, J)$ allows us to define a
        $\V$-functor $F: \G K \to \J$ by $FX = \alpha_X(X)$. Indeed,
        given $\V$-categories $X, Y$ we may consider any $Z$ into
        which $X, Y$ may be fully embedded (for example, their
        coproduct in $\VCat$) and obtain
        \begin{eqnarray*}
            1_{KZ}^*(X,Y) & \leq & J\big( \alpha_Z(X), \alpha_Z(Y)
            \big) \\
            & \leq & J\big( \alpha_X(X), \alpha_Y(Y) \big) \\
            &=& J (FX, FY).
        \end{eqnarray*}
        Hence, $F$ is indeed a $\V$-functor with $F\kappa_X =
        \alpha_X$ for all $X$, and it is obviously the only such
        $\V$-functor.
    \end{proof}
    For the sake of completeness we remark that the assignment
    \[
        K \mapsto \G K
    \]
    is monotone (=functorial): if we order subliftings of the
    powerset functor by
    \[
        K \leq L \iff \text{ there is a nat. tr. $\alpha: K \to L$
        given by inclusion functions},
    \]
    while $\V$-category structures on $\G = \ob(\VCat)/\cong$ carry
    the pointwise order (as $\V$-modules), then
    \[
        \G: \Sub H \to \VV{\CAT}(\G)
    \]
    becomes monotone. Indeed, for every $\V$-module $\varphi: Z
    \modto Y$, naturality of $\alpha$ gives
    \begin{eqnarray*}
        \alpha^*_Y \cdot \tilde{L}\varphi \cdot (\alpha_X)_* & = &
        \alpha^*_Y \cdot (L\yy_{\varphi})^* \cdot
        (L\yy_X)_*\cdot(\alpha_X)_* \\
        & = & (K \yy_{\varphi})^* \cdot \alpha_{\hat{X}}^*
        \cdot(\alpha_{\hat{X}})_* \cdot (K\yy_X)_*\\
        & \geq & (K\yy_{\varphi})^* \cdot (K \yy_X)_* =
        \tilde{K}\varphi.
    \end{eqnarray*}
    Consequently,
    \begin{eqnarray*}
        \tilde{K}\varphi(X,Y) & \leq & ( \alpha_Y^* \cdot
        \tilde{L}\varphi \cdot (\alpha_X)_*) (X,Y) \\
        & = &
        \tilde{L}\varphi(\alpha_X(X), \alpha_Y(Y)) \\
        & = & \tilde{L} \varphi(X,Y),
    \end{eqnarray*}
    which gives $GK(X,Y) \leq GL (X,Y)$, for all $\V$-categories $X,
    Y$.

\section{Operations on the Gromov-Hausdorff $\V$-category}

\begin{prop}
\label{p: GH_monoid}
With the binary operation $(X,Y)\mapsto X\otimes Y$ the $\V$-category $\G H$ becomes a monoid in the monoidal category $\VV\CAT$.
\end{prop}

\begin{proof}
All we need to show is that
\[\otimes:\G H\otimes\G H\to \G H\]
is a $\V$-functor. But for any $\V$-modules $\varphi:X\modto X'$, $\psi:Y\modto Y'$ and all $x\in X$, $y\in Y$ one trivially has
\[\tilde{H}\varphi(X,X')\otimes\tilde{H}\psi(Y,Y')\leq \bigvee_{x'\in X',y'\in Y'}\varphi(x,x')\otimes\psi(y,y'),\]
hence
\[\tilde{H}\varphi(X,X')\otimes\tilde{H}\psi(Y,Y')\leq\tilde{H}(\varphi\otimes\psi)(X\otimes Y,X'\otimes Y'),\]
with the $\V$-module $\varphi\otimes\psi:X\otimes Y\modto X'\otimes Y'$ given by
\[(\varphi\otimes\psi)((x,y),(x',y'))=\varphi(x,x')\otimes\psi(y,y').\]
Consequently,
\begin{eqnarray*}
GH\otimes GH((X,Y),(X',Y'))&=&GH(X,X')\otimes GH(Y,Y')\\
&=&\bigvee_{\varphi,\psi}\tilde{H}\varphi(X,X')\otimes \tilde{H}\psi(Y,Y')\\
&\leq&\bigvee_{\chi:X\otimes Y\submodto X'\otimes Y'}\tilde{H}\chi(X\otimes Y,X'\otimes Y')\\
&=&GH(X\otimes Y, X'\otimes Y').\end{eqnarray*}
\end{proof}

We note that when the $\otimes$-neutral element $k$ of $\V$ is its top element $\top$, then $v\otimes w\leq v\wedge w$ for all $v,w\in \V$ (since $v\otimes w\leq v\otimes k=v$); conversely, this inequality implies $k=\top$ (since $\top=\top\otimes k\leq \top\wedge k=k$).

\begin{prop}\label{p: mmonoid}
If $k=\top$ in $\V$, then $\G H$ becomes a monoid in the monoidal category $\VV\CAT$ with the binary operation given either by product or by coproduct.
\end{prop}

\begin{proof}
We need to show that
\[\times:\G H\otimes \G H\to\G H\;\mbox{ and }\;+:\G H\otimes \G H\to\G H\]
are both $\V$-functors. Similarly to the proof of Proposition \ref{p: GH_monoid}, for the $\V$-functoriality of $\times$ it suffices to show
\begin{equation}\label{eq: times}
\tilde{H}\varphi(X,X')\otimes\tilde{H}\psi(Y,Y')\leq \tilde{H}(\varphi\times\psi)(X\times Y, X'\times Y')\end{equation}
for all $\V$-modules $\varphi:X\modto X'$, $\psi:Y\modto Y'$, where $\varphi\times\psi:X\times Y\to X'\times Y'$ is defined by
\[(\varphi\times\psi)((x,y),(x',y'))=\varphi(x,x')\wedge \psi(y,y').\]
(Note that, in this notation, $1_X^*\times 1_Y^*$ is the $\V$-category structure of the product $X\times Y$ in $\VV\Cat$. The verification that $\varphi\times\psi$ is indeed a $\V$-module is easy.)
But (\ref{eq: times}) follows just like in Proposition \ref{p: GH_monoid} since $k=\top$.

For the $\V$-functoriality of $+$ it suffices to establish the inequality
\begin{equation}\label{eq: plus}
\tilde{H}\varphi(X,X')\otimes \tilde{H}\psi(Y,Y')\leq\tilde{H}(\varphi+\psi)(X+Y,X'+Y'),\end{equation}
with $\varphi+\psi:X+Y\modto X'+Y'$ defined by
\[(\varphi+\psi)(z,z')=\left\{\begin{array}{ll}
\varphi(z,z')&\mbox{ if }z\in X,\,z'\in X',\\
\psi(z,z')&\mbox{ if }z\in Y, z'\in Y',\\
\bot&\mbox{ else.}
\end{array}\right.\]
(Again, $1_X^*+1_Y^*$ is precisely the $\V$-category structure of the coproduct $X+Y$ in $\VV\Cat$, and the verification of the $\V$-module property of $\varphi+\psi$ is easy.) To verify (\ref{eq: plus}) we consider $z\in X+Y$; then, for $z\in X$, say, we have
\begin{eqnarray*}
\tilde{H}\varphi(X,X')\otimes\tilde{H}\psi(Y,Y')&\leq&\tilde{H}\varphi(X,X')\wedge\tilde{H}\psi(Y,Y')\\
&\leq&\tilde{H}\varphi(X,X')\\
&\leq&\bigvee_{x'\in X'}\varphi(z,x')\\
&\leq&\bigvee_{z'\in X'+Y'}(\varphi+\psi)(z,z'),\end{eqnarray*}
and (\ref{eq: plus}) follows.
\end{proof}

The previous proof shows that, without the assumption $k=\top$, one has that $+:\G H\times\G H\to\G H$ is a $\V$-functor, e.g. that $(\G H,+)$ is a monoid in the Cartesian category $\VV\CAT$, but here we will continue to consider the monoidal structure of $\VV\CAT$.

\begin{theorem}\label{t: Haus functor}
If $k=\top$ in $\V$, then the Hausdorff functor $H:\VV\Cat\to\VV\Cat$ induces a homomorphism $H:(\G H, +)\to(\G H, \times)$ of monoids in the monoidal category $\VV\CAT$.
\end{theorem}

\begin{proof}
Let us first show that the object-part of the functor $H:\VV\Cat\to\VV\Cat$ defines indeed a $\V$-functor $H:\G H\to\G H$, so that $GH(X,Y)\leq GH(HX,HY)$ for all $\V$-categories $X,Y$. But for every $\V$-module $\varphi:X\modto Y$ and all $A\subseteq X$ one has
\[\tilde{H}\varphi(X,Y)\leq\tilde{H}\varphi(A,Y)\leq\bigvee_{B\subseteq Y}\tilde{H}\varphi(A,B),\]
which implies
\[\tilde{H}\varphi(X,Y)\leq\tilde{H}(\tilde{H}\varphi)(HX,HY)\]
and then the desired inequality.

In order to identify $H$ as a homomorphism, we first note that, as an empty meet, $h_\varnothing(\varnothing,\varnothing)$ is the top element in $\V$, so that $H\varnothing\cong 1$ is terminal in $\VV\Cat$, e.g. neutral in $(\G H, \times)$. The bijective map
\[+:HX\times HY\to H(X+Y)\]
needs to be identified as an isomorphism in $\VV\Cat$, e.g. we must show
\[(h_X\times h_Y)((A,B),(A',B'))=h_{X+Y}(A+B,A'+B')\]
for all $A,A'\subseteq X$, $B, B'\subseteq Y$. With $a=1_X^*$ and $b=1_Y^*$, in the notation of the proof of Proposition \ref{p: mmonoid} one has
\[\bigvee_{z'\in A'+B'}(a+b)(x,z')=\bigvee_{x'\in A'}a(x,x')\]
for all $x\in A$ (since $(a+b)(x,z')=\bot$ when $z'\in B'$). Consequently,
\begin{eqnarray*}
h_{X+Y}(A+B,A'+B')&=&(\bigwedge_{x\in A}\bigvee_{z'\in A'+B'}(a+b)(x,z'))\wedge(\bigwedge_{y\in B}\bigvee_{z'\in A'+B'}(a+b)(y,z'))\\
&=&(\bigwedge_{x\in A}\bigvee_{x'\in A'}a(x,x'))\wedge(\bigwedge_{y\in B}\bigvee_{y'\in B'}b(y,y'))\\
&=&h_X(A,A')\wedge h_Y(B,B'),\end{eqnarray*}
as desired.
\end{proof}

\begin{remarks}\label{remarks_four}
\begin{enumerate}
\item The ($\VV\Cat$)-isomorphism
\[HX\times HY\cong H(X+Y)\]
exhibited in the proof of Theorem \ref{t: Haus functor} easily extends to the infinite case:
\[\prod_{i}HX_i\cong H(\sum_i X_i).\]
\item Since there is no general concept of a (covariant!) functor transforming coproducts into products, a more enlightening explanation for the formula just encountered seems to be in order, as follows. Since $\VV\Cat$ is an {\em extensive category} (see \cite{CLW}), for every (small) family $(X_i)_{i\in I}$ of $\V$-categories the functor
    \[\Sigma:\prod_i\VV\Cat/X_i\to\VV\Cat/\sum_i X_i\]
    is an equivalence of categories. Now, the (isomorphism classes of a) comma category $\VV\Cat/X$ can be made into a (large) $\V$-category when we define the $\V$-category structure $c$ by
    \[c(f,g)=\bigwedge_{x\in A}\bigvee_{y\in B} 1_X^*(f(x),g(y))=h_X(f(A),g(B)),\]
    for all $f:A\to X$, $g:B\to X$ in $\VV\Cat$. In this way the equivalence $\Sigma$ has become an isomorphism of $\V$-categories, and since $HX$ is just a $\V$-subcategory of $\VV\Cat/X$, the ($\VV\Cat$)-isomorphism of (1) is simply a restriction of the isomorphism $\Sigma$:
    \[\xymatrix{\prod_i\VV\Cat/X_i\ar[r]^\sum&\VV\Cat/\sum_iX_i\\
    \prod_iHX_i\ar@{^(->}[u]\ar[r]^\sim&H(\sum_iX_i)\ar@{^(->}[u]}\]
\end{enumerate}
\end{remarks}

\section{Symmetrization}
A $\V$-category $X$, or just its structure $a=1_X^*$, is {\em symmetric} when $a=a^\circ$. This defines the full subcategory $\VV\Cat_s$ of $\VV\Cat$ which is coreflective: the coreflector sends an arbitrary $X$ to $X^s=(X,a^s)$ with $a^s=a\times a^\circ$, that is: $a^s(x,y)=a(x,y)\wedge a(y,x)$ for all $x,y\in X$. By
\[H^s X=(HX)^s=(PX,h_X^s)\]
one can define a sublifting $H^s:\VV\Cat\to\VV\Cat$ of the powerset functor which (like $H$) preserves full fidelity, but which (unlike $H$) fails to be a 2-functor. However its restriction
\[H^s:\VV\Cat_s\to\VV\Cat_s\]
is a 2-functor.

\begin{remarks}\label{remarks_five}
\begin{enumerate}
\item $H^sX$ must not be confused with $H(X^s)$. For example, for $\V=\two$ and a set $X$ provided with a separated (=antisymmetric) order, $X^s$ carries the discrete order. Hence, while in $H^sX$ one has ($A\leq B\iff A\subseteq \downarrow B$ and $B\subseteq\downarrow A\iff \downarrow A=\downarrow B$), in $H(X^s)$ one has ($A\leq B\iff A\subseteq B$).
    \item Even after its restriction to $\VV\Cat_s$ there is no easy way of evaluating $\widetilde{H^s}\varphi(A,B)$ for a $\V$-module $\varphi:X\modto Y$ and $A\subseteq X$, $B\subseteq Y$, since the computation leading to the easy formula of Lemma \ref{lemma_two} does not carry through when $H$ is replaced by $H^s$.
        \item The following addendum to Proposition \ref{prop_three} suggests how to overcome the difficulty mentioned in (2) when trying to define a non-trivial symmetric Gromov structure: $\V$-category structures $c$ on the disjoint union $X\sqcup Y$ such that the $\V$-categories $X, Y$ become full $\V$-subcategories correspond bijectively to pairs of $\V$-modules $\varphi:X\modto Y$, $\varphi':Y\modto X$ with
            \[\varphi'\cdot \varphi\leq 1_X^*,\;\varphi\cdot\varphi'\leq 1_Y^*;\]
            we write
            \[\xymatrix{\varphi:X\ar@<0.5ex>[r]|{\circ}&Y:\varphi'\ar@<0.5ex>[l]|{\circ}}\]
            for such a pair. Under the hypotheses of Proposition \ref{prop_three} we can now write
            \[GK(X,Y)=\bigvee_{\xymatrix@-10pt{\varphi:X\ar@<0.5ex>[r]|{\circ}&Y:\varphi'\ar@<0.5ex>[l]|{\circ}}} \tilde{K}\varphi(X,Y).\]
            \end{enumerate}
            \end{remarks}

Hence, for any sublifting $K$ of $P$ we put
\[G^sK(X,Y):=\bigvee_{\xymatrix@-10pt{\varphi:X\ar@<0.5ex>[r]|{\circ}&Y:\varphi'\ar@<0.5ex>[l]|{\circ}}} \tilde{K}\varphi(X,Y)\wedge \tilde{K}\varphi'(Y,X)\]
and obtain easily:

\begin{corollary}\label{coro_three}
For any sublifting $K$ of the powerset functor,
\[\G^sK=(\G,G^sK)\]
is a large symmetric $\V$-category, and when $K$ is a 2-functor preserving full fidelity of $\V$-functors, then
\[G^sK(X,Y)=\bigvee_{X\hookrightarrow Z\hookleftarrow Y} 1_{KZ}^*(X,Y)\wedge 1_{KZ}^*(Y,X)=\bigvee_{X\hookrightarrow (X\sqcup Y,c)\hookleftarrow Y} \tilde{K}c(X,Y)\wedge \tilde{K}c(Y,X)\]
for all $\V$-categories $X, Y$.
\end{corollary}
\begin{proof}
Revisiting the proof of Proposition \ref{prop_two}, we just note that
\[\xymatrix{\varphi:X\ar@<0.5ex>[r]|{\circ}&Y:\varphi'\ar@<0.5ex>[l]|{\circ}}, \xymatrix{\psi:Y\ar@<0.5ex>[r]|{\circ}&Z:\psi'\ar@<0.5ex>[l]|{\circ}} \mbox{ implies }
\xymatrix{\psi\cdot\varphi:X\ar@<0.5ex>[r]|{\circ}&Z:\varphi'\cdot\psi'\ar@<0.5ex>[l]|{\circ}}.\]
A slight adaption of the computation given in Proposition
\ref{prop_two} now shows that $G^sK$ is indeed a $\V$-category
structure on $\G=\ob\VV\Cat/\cong$. The given formulae follow as in
the proof of Proposition \ref{prop_three}.
\end{proof}

\begin{corollary}\label{coro_four}
$G^sH(X,Y)=G(H^s)(X,Y)$, for all $\V$-categories $X, Y$.
\end{corollary}

Extending the notion of symmetry from $\V$-categories to $\V$-modules, we call a $\V$-module $\varphi:X\modto Y$ {\em symmetric} if $X, Y$ are symmetric with $\varphi^\circ\cdot\varphi\leq 1_X^*$ and $\varphi\cdot\varphi^\circ\leq 1_Y^*$; we write
\[\xymatrix{\varphi:X\ar@<0.5ex>[r]|{\circ}&Y:\varphi^\circ\ar@<0.5ex>[l]|{\circ}}\]
in this situation and define
\[G_sK(X,Y):=\bigvee_{\xymatrix@-10pt{\varphi:X\ar@<0.5ex>[r]|{\circ}&Y:\varphi^\circ\ar@<0.5ex>[l]|{\circ}}} \tilde{K}\varphi(X,Y)\]
for every sublifting $K$ of $P$. Since symmetric $\V$-modules compose, similarly to Corollary \ref{coro_three} one obtains:

\begin{corollary}\label{coro_five}
For any sublifting $K$ of the powerset functor,
\[\G_sK:=(\ob\VV\Cat_s/\cong, G_sK)\]
is a large $\V$-category, and when $K$ is a 2-functor preserving full fidelity of $\V$-functors, then
\[G_sK(X,Y)=\bigvee_{\substack{X\hookrightarrow Z\hookleftarrow Y\\ Z\;\mbox{\tiny symmetric}}} 1_{KZ}^*(X,Y)=\bigvee_{\substack{X\hookrightarrow (X\sqcup Y,c)\hookleftarrow Y\\c\;\mbox{\tiny symmetric}}} \tilde{K}c(X,Y)\]
for all symmetric $\V$-categories $X,Y$.
\hfill$\Box$
\end{corollary}

\begin{remarks}\label{remarks_six}
\begin{enumerate}
\item It is important to note that $G_sK$ is {\em not} symmetric, even when $K=H$. For $\V=\PP_+$, $X$ a singleton and $Y$ 3 equidistant points, we already saw in the Introduction that $G_sH(X,Y)=0$ while $G_sH(Y,X)=\frac{1}{2}$. Hence it is natural to consider the symmetrization $(G_sK)^s$ of $G_sK$:
    \[(G_sK)^s(X,Y)=G_sK(X,Y)\wedge G_sK(X,Y).\]
    The same example spaces of the Introduction show that, while $(GH)^s(X,Y)=\max\{GH(X,Y),GH(Y,X)\}=0,$ one has
    \[(G_sH)^s(X,Y)=\max\{G_sH(X,Y), G_sH(Y,X)\} = \frac{1}{2}.\]
    \item When the {\em symmetric} $\V$-categories $X, Y$ are fully embedded into some $\V$-category $Z$, they are also fully embedded into $Z^s$. This fact gives
        \[G^sH(X,Y)\leq G_sH(X,Y)\]
        which, by symmetry of $G^sH$, gives
        \[G^sH(X,Y)\leq (G_sH)^s(X,Y).\]
        \item Instead of the coreflector $X\mapsto X^s$ one may consider the {\em monoidal symmetrization} $X^\sym=(X,a^\sym)$ with $a^\sym=a\otimes a^\circ$, that is: $a^\sym(x,y)=a(x,y)\otimes a(y,x)$. Hence, replacing $\wedge$ by $\otimes$ one can define $H^\sym X$ and $G^\sym K$ in complete analogy to $H^sX$ and $G^sX$, respectively. Corollary \ref{coro_three} remains valid when $s$ is traded for $\sym$ and $\wedge$ for $\otimes$.
            \end{enumerate}
            \end{remarks}

\section{Separation}

A $\V$-category $X$, or just its structure $a=1_X^*$, is {\em separated} when $k\leq a(x,y)\wedge a(y,x)$ implies $x=y$ for all $x,y\in X$. It was shown in \cite{HT} (and it is easy to verify) that the separated $\V$-categories form an epireflective subcategory of $\VV\Cat$: the image of $X$ under the Yoneda functor $\yy_X:X\to\hat{X}$ serves as the reflector. Furthermore, there is a closure operator which describes separation of $X$ equivalently by the closedness of the diagonal in $X\times X$. (This description is not needed in what follows, but it further confirms the naturality of the concept.)

In Remarks \ref{remarks_two} we already presented a sublifting
$H_\Downarrow$ of the powerset functor, and it is easy to check that
$\widetilde{H_\Downarrow}\varphi(A,B)$ may be computed as
$\tilde{H}\varphi(A,B)$ in Lemma \ref{lemma_two}, e.g. the two
values coincide, because of the formula proved in Remarks
\ref{remarks_two}(1). Furthermore, $H_\Downarrow$ is like $H$ a
2-functor which preserves full fidelity of $\V$-functors. Hence,
Proposition \ref{prop_three} is applicable to $H_\Downarrow$ and may
in fact be sharpened to:

\begin{corollary}\label{coro_six}
For all separated $\V$-categories $X, Y$ one has
\[GH(X,Y)=GH_\Downarrow (X,Y)=\bigvee_{\substack{X\hookrightarrow Z\hookleftarrow Y\\ Z\;\mbox{\tiny separated}}} h_Z(X,Y)=\bigvee_{\substack{X\hookrightarrow (X\sqcup Y,c)\hookleftarrow Y\\c\;\mbox{\tiny separated}}}
\tilde{H}c(X,Y).\]
\end{corollary}
\begin{proof}
The structure $c$ constructed from a $\V$-module $\varphi$ as in the proof of Proposition \ref{prop_three} is separated.
\end{proof}

\begin{remarks}
\begin{enumerate}
\item From Corollary \ref{coro_three} one obtains
\begin{eqnarray*}
G^sH(X,Y)=G^sH_\Downarrow (X,Y)=&\displaystyle\bigvee_{X\hookrightarrow Z\hookleftarrow Y} h_Z(X,Y)\wedge h_Z(Y,X)\vspace*{2ex}\\
=&\displaystyle\bigvee_{X\hookrightarrow (X\sqcup Y,c)\hookleftarrow Y} \tilde{H}c(X,Y)\wedge \tilde{H}c(Y,X).\end{eqnarray*}
However, here it is {\em not} possible to restrict the last join to separated structures $c$: consider the trivial case when $\V=\two$ and $X,Y$ are singleton sets.
\item $\V$-category structures $c$ on $X\sqcup Y$ that are both symmetric and separated correspond bijectively to symmetric modules $\varphi:X\modto Y$ with $k\not\leq \varphi(x,y)$ for all $x\in X$, $y\in Y$, provided that $X$ and $Y$ are both symmetric and separated. For $\V=\two$, $X, Y$ are necessarily discrete, and the only structure $c$ is
discrete as well.
    \item The structure $GH$ on $\G$ is {\em not} separated, even if we consider only isomorphism classes of separated $\V$-categories: for $\V=\two$, the order on $\G$ given by $GH$ is chaotic! Likewise when $G$ is traded for $G^s$.
        \end{enumerate}
        \end{remarks}


\begin{thebibliography}{99}

\bibitem{AA} A.\ Akhvlediani, {\em Hausdorff and Gromov distances in quantale-enriched categories}, Master's Thesis, York University, 2008.

\bibitem{BBI} D.\ Burago, Y.\ Burago and S.\ Ivanov, {\em A Course
in Metric Geometry}, American Math. Society, Providence, R. I.,
2001.

\bibitem{CH04} M.M.\ Clementino and D.\ Hofmann, On extensions of
lax monads, {\em Theory Appl.\ Categ.} {\bf 13} (2004) 41--60.

\bibitem{CH07} M.M.\ Clementino and D.\ Hofmann, Lawvere completeness in topology,
{\em Appl.\ Categ.\ Structures} (to appear).

\bibitem{CHT} M.M.\ Clementino, D.\ Hofmann and W.\ Tholen,
One setting for all: metric, topology, uniformity, approach
structure, {\em Appl.\ Categ.\ Structures} {\bf 12} (2004) 127--154.

\bibitem{CLW} A.\ Carboni, S.\ Lack and R.F.C.\ Walters,
Introduction to extensive and distributive categories, {\em J.\ Pure Appl.\ Algebra}  {\bf 84}  (1993) 145--158.

\bibitem{G1} M.\ Gromov, Groups of polynomial growth and expanding maps, {\em Publications math\'{e}matiques de l'I.H.E.S.},
{\bf 53} (1981) 53-78.

\bibitem{G2} M.\ Gromov, {\em Metric Structures for Riemannian and Non-Riemannian Spaces}, Birkh\"auser, Boston, MA, 2007.

\bibitem{Ha} F. Hausdorff, {\em Grundz\"uge der Mengenlehre},
Teubner, Leipzig 1914.

\bibitem{HT} D.\ Hofmann and W.\ Tholen, Lawvere completion and separation via closure,
{\em Appl.\ Categ.\ Structures} (to appear).

\bibitem{Ke} G.M.\ Kelly, {\em Basic concepts of enriched category theory}, volume~64 of
{\em London Mathematical Society Lecture Note Series}, Cambridge University Press, Cambridge, 1982.

\bibitem{L} F.W.\ Lawvere, Metric spaces, generalized logic, and closed
categories, {\em Rend.\ Sem.\ Mat.\ Fis.\ Milano} {\bf 43} (1973)
135--166; {\em Reprints in Theory and Applications of Categories}
{\bf 1} (2002) 1--37.

\bibitem{Po} D.\ Pompeiu, {\em Sur la continuite des fonctions de variables complexes}, Doctoral thesis, Paris, 1905.

\bibitem{St} I. Stubbe, ``Hausdorff distance" via conical cocompletion,
{\em preliminary report}, December 2008.

\bibitem{Wood} R.J.\ Wood, Ordered sets via adjunctions, in: {\em Categorical Foundations},
volume~97 of {\em Encyclopedia Math. Appl.}, pages 5--47. Cambridge Univ. Press, Cambridge, 2004.

\end{thebibliography}
\end{document}